\documentclass[11pt,twoside,a4paper]{amsart}
\usepackage{amssymb}
\date{\today}

\title{Regularity of pseudomeromorphic currents}

\author{Mats Andersson \& Elizabeth Wulcan}
\thanks{The authors were partially supported by the Swedish Research Council.} 
\subjclass[2000]{32A26, 32A27, 32B15, 32C30}
\address{Department of Mathematical Sciences, Division of Mathematics, University of Gothenburg and 
Chalmers University of Technology, SE-412 96 G\"{o}teborg, Sweden}
\email{matsa@chalmers.se,  wulcan@chalmers.se}

\usepackage{geometry}
\newtheorem{thm}{Theorem}[section]
\newtheorem{lma}[thm]{Lemma}
\newtheorem{cor}[thm]{Corollary}
\newtheorem{prop}[thm]{Proposition}

\theoremstyle{definition}

\newtheorem{df}[thm]{Definition}

\theoremstyle{remark}

\newtheorem{preremark}[thm]{Remark}
\newtheorem{preex}[thm]{Example}

\newenvironment{remark}{\begin{preremark}}{\qed\end{preremark}}
\newenvironment{ex}{\begin{preex}}{\qed\end{preex}}

\newcommand{\ff}{\varphi}

\newcommand{\T}{\mathcal{T}}
\newcommand{\ASM}{\text{ASM}}
\newcommand{\F}{\mathcal{F}}
\newcommand{\cqa}{cqa}
\newcommand{\Hs}{\mathcal{H}}

\newcommand{\C}{\mathbb{C}}
\newcommand{\N}{\mathbb{N}}

\newcommand{\dbar}{\bar{\partial}}

\newcommand{\R}{\mathbb{R}}
\newcommand{\J}{\mathcal{J}}
\newcommand{\I}{\mathcal{I}}
\newcommand{\E}{\mathcal{E}}

\newcommand{\W}{\mathcal{W}}
\newcommand{\PM}{\mathcal{PM}}

\newcommand{\Q}{\mathbb{Q}}

\newcommand{\Ok}{\mathcal{O}}

\newcommand{\K}{\mathcal{K}}
\newcommand{\D}{\mathcal{D}}
\newcommand{\V}{\mathcal {V}}

\newcommand{\CH}{\mathcal{{ CH}}}

\newcommand{\M}{\mathcal{M}}
\newcommand{\Ak}{\mathcal{A}}
\newcommand{\pmm}{pseudomeromorphic }
\newcommand{\nbh}{neighborhood }
\newcommand{\1}{{\bf 1}}
\newcommand{\w}{{\wedge}}
\newcommand{\codim}{{\text{codim}\,}}

\newcommand{\Homs}{{\mathcal Hom}}
\newcommand{\Hom}{{\text{Hom}\,}}
\newcommand{\End}{{\text{End}\,}}

\newcommand{\Exts}{{\mathcal Ext\,}}
\newcommand{\Kers}{{\mathcal Ker\,}}

\newcommand{\U}{{\mathcal U}}

\newcommand{\Cu}{{\mathcal C}}

\def\newop#1{\expandafter\def\csname #1\endcsname{\mathop{\rm #1}\nolimits}}
\newop{span}

\DeclareMathOperator{\rank}{rank}
\DeclareMathOperator{\supp}{supp}
\DeclareMathOperator{\sing}{sing}

\numberwithin{equation}{section}

\begin{document}
\nocite{*}
\bibliographystyle{plain}

\begin{abstract}
 Let $X$ be a (reduced) pure-dimensional analytic space. 
We prove that direct images of principal value and residue currents
on $X$ are smooth
outside sets that are small in a certain sense.  We also prove that 
the sheaf of such currents, provided that $X$ is smooth, is a stalkwise injective $\Ok_X$-module.
\end{abstract}

\maketitle


\section{Introduction}
 Let $f$ be a generically nonvanishing  holomorphic function on
 a reduced analytic space $X$ of pure dimension $n$.
Herrera and Lieberman, \cite{HeLi}, proved that the principal value
\begin{equation*}
\lim_{\epsilon\to 0}\int_{|f|^2>\epsilon}\frac{\xi}{f}
\end{equation*}
exists for test forms $\xi$ and defines a current $[1/f]$.
It follows that $\dbar[1/f]$ is a current with support on the zero set
$Z(f)$ of $f$; such a current is called a residue current. Coleff and
Herrera, \cite{CoHe}, introduced (non-commutative) products of principal value and
residue currents, like
\begin{equation}\label{apan2}
[1/f_1]\cdots [1/f_r] \dbar[1/f_{r+1}]\w\cdots 
\w\dbar[1/f_m].
\end{equation}
The theory of (products of) residue and principal value currents has been further developed
by a number of authors since then,  see, e.g., the references given in
\cite{AW3}.  
 
\smallskip
In order to obtain a coherent approach to questions about residue and
principal value currents were introduced in
\cite{AW2, AS} the sheaf  $\PM_X$ of {\it pseudomeromorphic currents} on $X$,  
consisting of  direct images under holomorphic mappings of products of test forms
and currents like \eqref{apan2}. 
%
Pseudomeromorphic currents play a decisive role in several recent papers, concerning, e.g., effective division problems and the
$\dbar$-equation on singular spaces; see \cite{AW3} for various references.

\smallskip
The objective of this paper is to study regularity properties of \pmm currents.
To understand the singular support of a \pmm current one is lead to
study non-proper images of analytic sets. 
Our first main result Theorem~\ref{tom} states that a \pmm current is smooth outside a set
that is small in a certain sense. 
 
Our second main result Theorem ~\ref{statut} asserts that $\PM_X$ 
is "ample" in the
sense that it is a stalkwise injective $\Ok_X$-module if $X$ is
smooth. 
The simplest
instance of this result  is that the equation $f\nu=\mu$ has a \pmm solution
for any \pmm current $\mu$ and nontrivial holomorphic function $f$. 
In particular this means that,  although smooth outside small sets,  \pmm currents can be quite singular.
The analogue of Theorem ~\ref{statut} for general currents 
is a classical result by Malgrange.

\smallskip
Combining Theorem ~\ref{statut} with the fact that $\PM_X^{0,\bullet}$ is a fine resolution of $\Ok_X$, which was noticed already in \cite{AS}, we obtain a generalization of the classical
Dickenstein-Sessa decomposition, \cite{DS},  in Section ~\ref{dickenstein}.

\smallskip

The proof of  Theorem ~\ref{statut} is based on an integral formula and
relies heavily on the regularity result
Theorem ~\ref{tom}.  Another  important ingredient is the fact from 
\cite{AW3} that one can "multiply" arbitrary \pmm currents
by 
proper direct images of principal value currents. 

\smallskip

In Section ~\ref{pmsection}  we recall some basic facts about  \pmm
currents and 
in Sections ~\ref{regprop} and
~\ref{taknock} we prove Theorem ~\ref{tom} and some variants. The last two sections are devoted
to discussion and proof of Theorem ~\ref{statut}.

\section{Pseudomomeromorphic currents}\label{pmsection}
In one complex variable $s$ one can define the principal value current $[1/s^m]$ for instance as the limit  
$$
\big[\frac{1}{s^m}\Big]=\lim_{\epsilon\to 0}\chi(|s|/\epsilon)\frac{1}{s^m},
$$
where $\chi:\R\to\R$ is a smooth function 
that is equal to $0$ in a neighborhood of
$0$ and $1$ in a neighborhood of $\infty$; we write $\chi\sim
\chi_{[1,\infty)}$ to denote such a $\chi$.  
We have the relations
\begin{equation}\label{utter}
\frac{\partial}{\partial s}\Big[\frac{1}{s^m}\Big]=-m\Big[\frac{1}{s^{m+1}}\Big], \quad s\Big[\frac{1}{s^{m+1}}\Big]=
\Big[\frac{1}{s^m}\Big].
\end{equation}
It is also well-known that
\begin{equation}\label{snok1}
\dbar\big[\frac{1}{s^{m+1}}\Big].\xi ds=\frac{2\pi i}{m!}\frac{\partial^m}{\partial s^m}\xi(0).
\end{equation}
for test functions $\xi$; in particular, $\dbar[1/s^{m+1}]$ has
support at $\{s=0\}$. It follows from \eqref{snok1} that 
\begin{equation}\label{snok2}
\bar s \dbar\big[\frac{1}{s^{m+1}}\Big]=0, \quad d\bar s\w \dbar\big[\frac{1}{s^{m+1}}\Big]=0.
\end{equation}

 \smallskip

Let $t_j$ be coordinates in an open set $\U\subset \C^N$ and let
$\alpha$ be a smooth form with compact support in $\U$.
Then 
\begin{equation}\label{1elem}
\tau=\alpha\w\Big[\frac{1}{t_1^{m_1}}\Big]\cdots\Big[\frac{1}{t_k^{m_k}}\Big]\dbar\Big[\frac{1}{t_{k+1}^{m_{k+1}}}\Big]\w\ldots
\w \dbar\Big[\frac{1}{t_r^{m_r}}\Big]
\end{equation}
is a well-defined current,
since it is the tensor product of
one-variable currents (times $\alpha$). We say that $\tau$ is an {\it elementary \pmm current}, and
we refer to $[1/t_j^{m_j}]$ and $\dbar[1/t_\ell^{m_\ell}]$ as its
{\it principal value factors} and  {\it residue factors}, respectively.
It is clear that \eqref{1elem} is commuting in the principal value factors and
anti-commuting in the residue factors.
We say the the intersection of $\U$ and the coordinate plane
$\{t_{k+1}=\cdots=t_r=0\}$ is the 
{\it elementary support} of $\tau$. Clearly the support of $\tau$ is  contained in
the intersection of the elementary support and the support of
$\alpha$. 

\begin{remark}\label{skymma} 
In view of \eqref{utter}, notice that $\partial \tau$ is an elementary
current, whose elementary support either equals the elementary
support $H$ of $\tau$ or is empty. Also $\dbar\tau$ is a finite sum of
elementary currents, whose elementary supports are either
equal to $H$ or coordinate planes of codimension $1$ in $H$, cf.,  \eqref{utter} and \eqref{snok1}. 
\end{remark}

\subsection{Definition and basic properties}\label{basal}
Let $X$ be a reduced analytic space of pure dimension $n$. 
If $X$ is smooth we say that  germ $\mu$  of a current at $x\in X$ is {\it pseudomeromorphic} at $x$,
$\mu\in \PM_x$,  if it is a finite sum
of currents of the form
\begin{equation*}
\mu=\ff_*\tau,
\end{equation*}
where $\ff\colon \U'\to \U$ is a holomorphic mapping, $\U$ is a \nbh of
$x$, and $\tau$ is elementary in
$\U'\subset \C^m$.
By definition the union $\PM=\PM_X=\cup_x\PM_x$ is an open subset (of
the \'etal\'e space) of the sheaf $\Cu=\Cu_X$
of currents, and hence it is a subsheaf which we call the sheaf of {\it pseudomeromorphic}  
currents.  It follows from \cite[Theorem 2.15]{AW3} that this
definition is equivalent to the definition given in \cite{AS,AW3}
\footnote{The definition of pseudomeromorphic currents in \cite{AW2} was slightly more restrictive.}. 
Thus a section $\mu$ of $\PM$ in an open set $\V\subset X$,  $\mu\in\PM(\V)$, 
can be written as a locally finite sum 
\begin{equation}\label{batting}
\mu=\sum (\ff_\ell)_*\tau_\ell,
\end{equation}
where each $\ff_\ell$ is holomorphic and each $\tau_\ell$ is elementary. For simplicity
we will always suppress the subscript $\ell$ in $\ff_\ell$. 

\smallskip
If $X$ is a general 
analytic pure-dimensional space and $\pi\colon Y\to X$ is a smooth modification, then $\PM_X$ consists of all direct images of 
currents in $\PM_Y$.   It follows from \cite[Theorem 2.15]{AW3} 
that the sheaf so obtained is independent of the choice of $Y$. 
Thus we again have a representation \eqref{batting}, where in this case each 
$\ff_\ell$ is a holomorphic mapping into a smooth manifold composed by
a modification.
 
\begin{remark}\label{gryning}
Note that each elementary current $\tau$ is a finite sum of currents
$\tau_\ell$ such that the support of $\tau_\ell$ is contained in an
irreducible component of the elementary support of $\tau$. We may
therefore assume that each $\tau_\ell$ in \eqref{batting} has
irreducible elementary support. 
\end{remark}

From \cite[Corollary 2.16]{AW3} we have
 
\begin{lma}\label{grus1} 
Assume that  $\ff\colon W\to X$ is a holomorphic mapping and $X$ is smooth, or
$\ff$ is a composition of a mapping into a smooth manifold composed by 
a modification.
 If 
$\mu$ is \pmm in $W$ with compact support, then $\ff_*\mu$ is \pmm in $X$.
\end{lma}
 
Notice that if $\xi$ is a smooth form, then
\begin{equation}\label{kondor}
\xi \w \ff_*\mu=\ff_*(\ff^*\xi\w\mu).
\end{equation}
Applying \eqref{kondor} to the representation \eqref{batting} we see that
 $\PM_X$ is 
closed under exterior multiplication by smooth forms, since this is
true for elementary
currents.
For the same reason 
$\PM_X$ is closed under  
$\dbar$ and $\partial$,  
cf., Remark ~\ref{skymma}.

 Another important property that is inherited from elementary currents,
cf., \eqref{snok2}, 
is the fact that 
\begin{equation}\label{polis}
\bar h \mu=0, \quad d\bar h\w\mu=0
\end{equation}
if $h$ is a
holomorphic function that vanishes on the support of the \pmm
current $\mu$.  This means in particular 
that the action of the current $\mu$ only involves holomorphic derivatives of test
forms.   From \eqref{polis} we get the   
{\it dimension principle:}

\smallskip\noindent 
{\it  If $\mu$ is \pmm of bidegree $(*,p)$ and has support on the
analytic variety $V$, where $\codim V>p$, then  $\mu=0$. }

\smallskip
Given an analytic subvariety $V$ of an open subset $\U\subset X$,
the natural restriction of a \pmm current $\mu$ to $\U\setminus V$
has a canonical extension to a \pmm current $\1_{X\setminus V}\mu$ 
in $\U$.  The following lemma is just Lemma~2.6 in \cite{AW3}:

\begin{lma} \label{3apsko}
Let $V$ be a subvariety of $\U\subset X$, let $h$ be a holomorphic tuple  in $\U$ whose
common zero set is precisely $V$, let $v$ be a smooth and nonvanishing
function, and let $\chi\sim\chi_{[1,\infty)}$.
For each \pmm current $\mu$ in $\U$ we have
\begin{equation*}
\1_{\U\setminus V}= \lim_{\epsilon\to 0} \chi(|h|^2v/\epsilon) \mu.
\end{equation*}
\end{lma}

Because of the factor $v$, the lemma holds just as well for a holomorphic section
$h$ of a Hermitian vector bundle.

\smallskip

 It follows that 
\begin{equation*}
\1_V\mu:=\mu-\1_{\U\setminus V}\mu
\end{equation*}
has support on $V$. It is proved in
\cite{AW2} that this operation extends to all constructible
sets and that 
\begin{equation}\label{skolgard}
\1_V\1_W\mu=\1_{V\cap W}\mu
\end{equation}
holds. 
If $\alpha$ is a smooth form, then
\begin{equation}\label{brutus1}
\1_V (\alpha\w\mu)=\alpha\w \1_V\mu.
\end{equation}
Moreover, 
if $\ff\colon W\to X$ is a  holomorphic mapping as in Lemma ~\ref{grus1}
and $\mu$ has compact support, then
 \begin{equation}\label{brutus2}
\1_V\ff_*\mu=\ff_*\big(\1_{\ff^{-1}V}\mu\big).
\end{equation}

We will need the following observation, which can proved in the same
way as Lemma~2.8 in \cite{AW3}, using \eqref{brutus2}.

\begin{lma}\label{batong}
If $\mu$ has the form \eqref{batting} then
$$
\1_V\mu=\sum_{\supp\tau_\ell\subset\ff^{-1}V} \ff_*\tau_\ell.
$$
One can just as well take the sum over all
$\ell$ such that
the elementary supports of $\tau_\ell$ are  contained in $\ff^{-1}V$. 
\end{lma}

For future reference we also include

\begin{lma}\label{1tensor}
If $T\in\PM_X$ and $T'\in\PM_{X'}$, then $T\otimes T'\in\PM_{X\times X'}$.
\end{lma}

See, e.g., \cite[Lemma~2.12]{AW3}. 
It is easy to verify that
\begin{equation}\label{pelargonia2}
\1_{V\times V'} T\otimes T'=\1_V T\otimes \1_{V'} T'.
\end{equation}

\subsection{The sheaves $\PM_X^Z$ and $\W_X^Z$}\label{regprop2}

Let $X$ be a reduced analytic space, let $Z$ be a (reduced) subspace
of pure dimension, and denote by 
$\PM_X^Z$ the subsheaf of $\PM_X$ of currents that have support on $Z$.
We say that $\mu\in\PM_X^Z$  has the {\it standard extension property}, SEP, on  
$Z$ if $\1_W\mu=0$ for each  
subvariety $W\subset \U\cap Z$ of positive codimension, where $\U$ is any open set in $X$.
Let $\W_X^Z$ be the subsheaf of $\PM_X^Z$ of currents with the SEP on
$Z$.  
If $Z=X$ we usually omit the superscript and just write
$\W$ or $\W_X$.

\begin{ex}\label{elex}
An elementary current in $\U\subset\C^n$ with elementary
support $H$ is in $\W_\U^H$. 
\end{ex}

\subsection{Almost semi-meromorphic currents}\label{asmsec}
The results and definitions in this and the next subsection are taken from \cite[Section 4]{AW3}. 
We say that a current on $X$ is \emph{semi-meromorphic} if it is of the form 
$\omega [1/f]$, where $f$ is a generically nonvanishing holomorphic
section of a line
bundle $L\to X$ and  $\omega$ is a smooth form with values in $L$. 
For simplicity we will often omit the brackets $[\  ]$ indicating principal value. 
Since $\omega [1/f]=[1/f]\omega$ when $\omega$ is smooth we can write just $\omega/f$.
   
\smallskip
Let $X$ be a pure-dimensional analytic space.  Following \cite{AS,
  AW3} we say that a current $a$ is {\it almost
semi-meromorphic} in $X$, $a\in ASM(X)$,  if there is a modification $\pi\colon X'\to X$ such that
\begin{equation}\label{asm}
a=\pi_*(\omega/f),
\end{equation}
where $\omega/f$ is semi-meromorphic in $X'$.
If $\U\subset X$ is an open subset, then the restriction $a_\U$ of $a\in ASM(X)$ to
$\U$ is in $ASM(\U)$.  Moreover, $ASM(X)$ is
contained in $\W(X)$.

Given a modification $\pi \colon X'\to X$, let $\sing(\pi)\subset X'$ be the (analytic) set 
where $\pi$ is not a
biholomorphism.  By the definition it has positive codimension.  Let $Z\subset X'$   be the zero set of $f$. Notice that $a\in ASM (X)$ is smooth outside
$\pi(Z\cup \sing(\pi))$, which has positive codimension in $X$. 
Let $ZSS(a)$, the {\it Zariski-singular support} of $a$, be the smallest 
Zariski-closed set $V\subset X$ such that $a$ is smooth outside $V$.

\begin{ex} If $f$ is a holomorphic function 
in $X$ such that $Z(f)$ has positive codimension, then clearly 
$[1/f]$ is almost semi-meromorphic and $ZSS(a)=Z(f)$.
\end{ex}

\begin{ex} \label{bex}
We claim that $b=\partial|\zeta|^2/2\pi i|\zeta|^2$ is
almost semi-meromorphic in $\C^n$.  In fact, let
$\pi\colon Y\to \C^n$ be the blow-up at the origin. Then, outside the
exceptional divisor,  
$\pi^*b=\omega/s$, where $s$ is a holomorphic section of the line bundle
$L_D$ that defines the exceptional divisor $D$ and $\omega$ is an $L_D$-valued
smooth $(1,0)$-form on $Y$. It is readily verified that
$b=\pi_*(\omega/s)$. In fact, it clearly holds outside the origin, and 
since both sides are locally integrable, the equality holds in the current sense.
Thus $b\in ASM(\C^n)$.
\end{ex}

We now recall one of the main results, Theorem 4.8,  in
\cite{AW3}:

\begin{thm}\label{hittills}
Assume that  $a\in ASM(X)$. For each $\mu\in\PM(X)$ 
there is a unique \pmm  current $T$ in $X$ that coincides with 
$a\w\mu$ in $X\setminus ZSS(a)$ and such that 
$\1_{ZSS(a)} T=0$.
\end{thm}

The proof is highly nontrivial and relies on the fact that
one can find a representation 
\eqref{asm} of $a$ such that $f$ is nonvanishing in $X'\setminus
\pi^{-1}ZSS(a)$ (\cite[Lemma 4.7]{AW3}).

Lemma ~\ref{3apsko} implies that
\begin{equation}\label{asm4}
T=\lim_{\epsilon\to 0}\chi(|h|^2v/\epsilon) a\w\mu
\end{equation}
if $h$ is a holomorphic tuple such that  $Z(h)=ZSS(a)$.
We will denote the extension $T$ by  $a\w\mu$ as well.

\smallskip

The definition of $a\wedge \mu$ is local, so that it commutes
with restrictions to open subsets of $X$.

\begin{prop} 
Assume that $a\in ASM(X)$.
If $W$ is an analytic subset of $\U\subset X$ and  $\mu\in\PM(\U)$, then
\begin{equation}\label{pelargonia}
\1_W(a\w\mu)=a\w \1_W\mu.
\end{equation}
\end{prop}

Clearly $\W_X^Z$ is closed under multiplication by smooth forms. We also have

\begin{prop} \label{koko}
Each $a\in ASM(X)$ induces a linear mapping
\begin{equation}\label{rav2}
\W_X^Z \to \W_X^Z, \quad
\mu\mapsto  a\wedge \mu.
\end{equation}
\end{prop}

\begin{prop}
Assume that $a_1, a_2\in ASM (X)$ and $\mu\in\PM_X$. Then 
\begin{equation*}
a_1\w a_2\w \mu=(-1)^{\deg a_1 \deg a_2} a_2\w a_1 \w \mu. 
\end{equation*}
\end{prop}

In particular, one of the $a_j$ may be a smooth form.  It follows that \eqref{rav2} is $\E$-linear.

\begin{ex}\label{skot} 
Assume that $\mu$ is in $\W$. In view of  \eqref{pelargonia},
$\mu'=[1/h]\mu$ is in $\W$ as well.
If $h$ is generically nonvanishing, then 
$h\mu'=h[1/h]\mu=\1_{\{h\neq 0\}}\mu=\mu$. 
\end{ex}

\subsection{Residues of almost semi-meromorphic currents}\label{lisa2}
We shall now study the effect of $\partial$ and $\dbar$ on almost semi-meromorphic currents.

\begin{prop}\label{skrot} 
If $a\in ASM(X)$, then $\partial a\in ASM(X)$ and 
\begin{equation}\label{trots}
\dbar a=b+r,
\end{equation}
where $b=\1_{X\setminus ZSS (a)} \dbar a$ is in $ASM(X)$  and 
$r=\1_{ZSS (a)} \dbar a$ has support on $ZSS(a)$. 
\end{prop}

Clearly the decomposition \eqref{trots} is unique. We call  $r=r(a)$ the
\emph{residue (current)} of $a$.

Notice that current $\dbar (1/f)$ is the residue of the
principal value current $1/f$.  
Similarly, the residue currents introduced, e.g., in \cite{PTY, A1, AW1} can be
considered as residues of certain almost semi-meromorphic currents,
generalizing $1/f$, cf.\ \cite[Example~4.18]{AW3}.

\smallskip 

As a consequence of Theorem ~\ref{hittills} we can define products 
of $\dbar$, and residues, of almost semi-meromorphic currents and pseudomeromorphic currents.  
\begin{df}\label{bounce}
For $a\in \ASM (X)$ and $\mu\in \PM_X$ we define 
\begin{equation}\label{konc}
\dbar a\w\mu:=\dbar(a\w\mu)-(-1)^{\deg a} a\w\dbar\mu,
\end{equation} 
where $a\w\mu$ and $a\w\dbar\mu$ are defined as in
Theorem ~\ref{hittills}. 
Moreover we define
\begin{equation*}
r(a)\wedge \mu:= \1_{ZSS(a)}\dbar a\wedge \mu. 
\end{equation*}
\end{df}
Thus $\dbar a\w\mu$ is defined so that  the Leibniz rule holds. 
It is easily checked that
\begin{equation}\label{dagis} 
r(a)\wedge\mu=\lim_{\epsilon\to 0}\dbar \chi(|h|^2v/\epsilon)a\w\mu, 
\end{equation} 
if $Z(h)=ZSS(a)$. 
%
%
In particular this gives a way of defining products of $\dbar$ and
residues of almost
semi-meromorphic currents. For example, \eqref{apan2} 
can be defined by inductively
applying \eqref{konc} and Theorem ~\ref{hittills}, cf.\ \cite{LS}.

\begin{ex}\label{bex2} 
Let $b$ be the almost semi-meromorphic current from
Example ~\ref{bex}.  If $n=1$, then $\dbar b$ is the current of integration $[0]$ at
the origin.  If $n>1$, then $\dbar b$ 
is almost semi-meromorphic since
then $r(b)$ must vanish in view of the dimension principle.  For 
$k\le n$ we can form the products
$B_k:= b\w (\dbar b)^{k-1}$.  It is just a product of
almost semi-meromorphic currents since no residues appear because of the dimension principle. However, it is well-known that
$\dbar B_n=[0]$. This is in fact a compact way of writing the Bochner-Martinelli
formula, see, e.g., \cite{Aint1}.
 \end{ex}

\section{Regularity of \pmm currents} \label{regprop}

We shall now discuss regularity properties of \pmm currents. To this end we first
have to consider local images of analytic sets under 
holomorphic mappings that are not necessarily proper.
Recall that if  $\ff:Y\to X$ is a  holomorphic mapping of between manifolds
and $Y$ is connected, then generically $\ff$ attains  its optimal
rank, $\rank \ff$, i.e.,
$\rank_y \ff=\rank \ff$ for all $y$ outside an analytic variety of positive codmension.

\begin{df}\label{cqadef}
Let $X$ be a complex manifold. We say that a compact set $V\subset X$ is a
\emph{\cqa\ } (compact quasianalytic set) if there is a (not
necessarily connected) complex manifold $Y$, a holomorphic map $\ff:Y\to
X$,  and a compact set $K\subset Y$, such that
$V=\ff(K)$. 
We say that the \emph{dimension} of $V$, $\dim V$, is $\leq d$ if $\rank_y \ff\leq d$ for all $y\in
K$. 
\end{df}

If $\dim V\leq d$, then the \emph{codimension} of $V$ is $\geq \dim
X-d$. 
%
%
If $d$ is as in Definition ~\eqref{cqadef} and $K$ has nonempty interior then we say that $\dim
V=d$. 

\begin{remark}
Our definition of a \cqa\ is closely related to the theory of subanalytic sets in the real setting, see, e.g.,
\cite{BM}. However we have not been able to rely directly on this theory. 
\end{remark}

\begin{ex}
Clearly, any compact set $K\subset X$ is a \cqa; however the 
dimension according to Definition~\ref{cqadef} might not be the expected. For example, in view of 
Example~\ref{nolldim} below, a point set with a
limit point cannot be a \cqa\  of dimension $0$. 
\end{ex}

Since we do not require $Y$ to be connected, any  finite union 
of \cqa s of dimension $\leq d$ is a \cqa\ of dimension $\leq d$.

\begin{remark}\label{inkl}
If $\ff:Y\to X$ is a holomorphic map of rank $n$, $X$ is a submanifold
of $M$, and $i:X\to M$ is the
inclusion, then $\rank i\circ \ff=n$. Thus if $V\subset X$ is a \cqa\ of dimension
$\leq n$, then so is $i(V)\subset M$.  
\end{remark}

\begin{remark}\label{vari}
We may allow $Y$ to be singular in Definition ~\ref{cqadef}. 
Indeed, assume that $V=\ff(K)$, where $\ff:Y\to X$ is a holomorphic map of
optimal rank $d$ and $Y$ is an analytic
variety. Let $\pi: \widetilde Y\to Y$
be a desingularization of $Y$. Then $\widetilde K:= \pi^{-1}(K)\subset\widetilde
Y$ is compact  and $\tilde \ff:= \ff\circ \pi:\widetilde Y\to X$ is a holomorphic map
of optimal rank $d$, and thus $V=\tilde \ff(\widetilde K)$ is a \cqa\ of
dimension $\leq d$ 
acccording to Definition ~\ref{cqadef}. 
\end{remark}

The notion of \cqa\  generalizes the notion of (a compact part of) a 
variety. 

\begin{ex}\label{vanlig}
Assume that $Z\subset X$ is a subvariety of pure dimension $\ell$. Then
$i:Z\to X$ has optimal rank $\ell$ and thus  any
compact $K\subset Z$ is
a \cqa\ of dimension $\leq \ell$. If $K$ has non-empty interior, then
$\dim K=\ell$. 
\end{ex}

There exists a \cqa\  that is not contained in an analytic 
variety of the same dimension.
The following example, which is a complex variant of an example
due to Osgood, see, e.g., \cite[Ex.\ 2.4]{BM}, was  
pointed out to us by Jean-Pierre Demailly. 

\begin{ex}\label{jpdex}
Let $u_1, u_2, u_3:\C\to \C$ be entire functions that are 
algebraically independent, e.g., let $u_i(z)=e^{a_iz}$, where
$a_1,a_2,a_3$ are linearly independent over $\Q$. 
Moreover let $\ff:\C^2\to \C^3$ be the map 
\[
\ff(z,w)=\big (u_1(z)w, u_2(z)w, u_3(z)w\big )
\]
and let $V=\ff(\overline\V)$, where $\V$ is a relatively compact \nbh of $0\in\C^2$.
Then  $V\subset \C^3$ is a \cqa\ of dimension $2$ since $\rank \ff=2$. 
We claim that $V$ is not contained in any $2$-dimensional subvariety
of an open set in $\C^3$ 
that contains $V$.
To prove this assume, to the contrary, that there is a holomorphic
function $g\not\equiv 0$ in a \nbh of $V$ such that $V\subset \{g=0\}$. 
Then $g(0)=0$. Let  
\[
g(x)=\sum_{m\in\N^3} a_{m}~x_1^{m_1} x_2^{m_2} x_3^{m_3}
\]
be the Taylor expansion of $g$ at $0\in\C^3$. Since $g\not\equiv 0$,
there is at least one index $m$ such that $a_{m}\neq 0$. 
Let $d$ denote the sum $m_1+m_2+m_3$ for this $m$. 
The assumption $V=\ff(\overline\V)\subset \{g=0\}$ implies that 
\[
0=g\circ \ff(z,w)=\sum_{m\in\N^3} a_{m}~u_1(z)^{m_1} u_2(z)^{m_2}
u_3(z)^{m_3} w^{m_1+m_2+m_3}
\]
for $(z,w)\in\V$. 
Identifying the coefficient of $w^d$ we get 
\[
\sum_{m_1+m_2+m_3=d}
a_{m}~u_1(z)^{m_1} u_2(z)^{m_2}
u_3(z)^{m_3} = 0,
\]
which contradicts the algebraic
independence of  $u_1$, $u_2$, and $u_3$ and thus proves the claim. 
\end{ex}

However, in a sense, a \cqa\ of dimension $\leq d$ is generically
contained in an analytic variety of dimension $d$.

\begin{lma}\label{caput2}
Assume that $V\subset X$ is a \cqa\  of dimension
$\leq d$. 
Then there is a \cqa\  $V'\subset V$ of dimension $\le d-1$,
such for each $x\in V\setminus V'$ there is a neighborhood $\U\subset
X$ of $x$ and a finite union $W\subset \U$ of submanifolds of
dimension $\leq d$ such
that $V\cap\U \subset W$. 
\end{lma}

If $d=0$, then $V'$ should be interpreted as the empty set; more
generally, a \cqa\ of dimension $\leq -1$ equals the empty set. 

\begin{proof}
Let  $V=\ff(K)$, where $\ff:Y\to X$ is a holomorphic map of
generic rank $\leq d$, $Y$ is a complex manifold, and $K\subset Y$ is
compact. 
Let $Y'=\{y\in Y, \rank_y \ff\le d-1\}$. Then $Y'$ is a subvariety of
$Y$,  and it follows,  cf., Remark ~\ref{vari}, that $V':=\ff(Y'\cap K)$ is a \cqa\ of dimension $\le d-1$.

If $V'=V$ the lemma is trivial. Otherwise,  
take  $x\in V\setminus V'$ and let $Z=\ff^{-1}(x)\cap K$.
If $y\in Z$ then $y\notin Y'$, and
since $Y'$ is closed there is a neighborhood $\V_y\subset Y$ of $y$ such
that $\ff$ has constant rank $d$ in $\V_y$. After possibly shrinking $\V_y$,
we may assume, in view of the constant rank theorem, that 
$\ff(\V_y)$ is a submanifold of dimension $d$ of some \nbh
$\U_y$ of $x$ in $X$.  By compactness, $Z$ is contained in a finite union
$\cup \V_{y_j}$ of such sets. Let $\U_{y_j}$ be the associated
neighborhoods of $x$. 

Since $K$ is compact and $\ff$ is continuous
there is
a neighborhood $\U\subset \cap \U_{y_j}$ of $x$ such that the closure of
$\ff^{-1}\U\cap K$ is contained in a finite union
$\cup \V_{y_j}$ of such sets $V_y$.  It follows that 
$V\cap\U$ is contained in 
$
W=\ff\big (\cup \V_{y_j}\big )\cap \U.
$
\end{proof}

\begin{ex}\label{nolldim}
It follows from Lemma ~\ref{caput2} that a \cqa\ of dimension $0$ is a
compact part of a variety of dimension $0$ and thus 
a discrete point set. 
\end{ex}

\begin{remark}
If $V=\ff(K)$, where $\ff:Y\to X$ has constant rank $d$, then $V'$ is empty in the
proof above, and thus $V$ is  contained in a subvariety of
$X$ of dimension $d$.
\end{remark}

\begin{ex}\label{jpdex2}
Let $\ff$ 
be as in Example ~\ref{jpdex}, with the choice $u_i(z)=e^{a_iz}$. 
Then 
\[
\frac{\partial \ff_i}{\partial z}=a_i e^{a_i z} w, ~~~~~
\frac{\partial \ff_i}{\partial w} = e^{a_i z}
\]
so it follows that $\rank_{(z,w)}=1$ if $w=0$ and $\rank_{(z,w)}=2$
otherwise. 
Thus, the set $Y'$ in the proof of Lemma ~\ref{caput2} equals $\{w=0\}$
and $V'=\ff(Y')=\{0\}$. Therefore the quasi-analytic set $V=\ff(\overline\V)$ is ``locally analytic'' outside $0$. 
\end{ex}

We have the following version of the dimension principle.

\begin{prop}\label{dimprin}
(i)\  If a \pmm current $\mu$ of bidegree $(*,p)$ has its support contained in 
a \cqa\ of codimension $\ge  p+1$, then $\mu =0$.
\smallskip

\noindent (ii)  \  If $\mu\in\W_X$ has support on a \cqa\ of positive codimension, then
$\mu=0$. 
\end{prop}

\begin{proof}
Assume that the support of $\mu$ is contained in the \cqa\ $V$ of codimension $\ge p+1$.
In view of Lemma ~\ref{caput2} and the usual dimension principle, see 
Section ~\ref{basal}, then $\mu$ must
have its support contained in a \cqa\ $V'$ of codimension $\ge p+2$.
Repeating the 
argument, (i) follows by a finite induction.  The statement (ii) is verified in a similar way.
 \end{proof}

\begin{ex}
Let us use the notation in Example ~\ref{jpdex}. 
Let $\chi$ be a cutoff function in $\C^2$ that is $1$ in a neighborhood of $0$ and 
$0$ outside $\V$ and let $\mu:=\ff_*\chi$. Then 
$$
\mu.1=\int_{\C^2} \chi \neq 0
$$
so $\mu$ is a \pmm nonvanishing current with compact support in the
\cqa\ $V$ in Example ~\ref{jpdex}. 
It follows from Proposition ~\ref{dimprin} ~(ii) that $\mu$ is not in
$\W_X$. 
However, note that $\1_{W}\mu=0$ for all germs of proper subvarieties $W$ at
$0\in\C^3$. In fact, $\1_{W}\mu=\ff_*(\1_{\ff^{-1}W}\chi)=0$ by the dimension principle, since
$\ff^{-1}W$ has positive codimension in $Y$ in view of Example ~\ref{jpdex}.
\end{ex}

We are now ready for our main result of this section.

\begin{thm}\label{tom}
Let $\mu$ be a pseudomeromorphic current with compact support on a manifold $X$
of dimension $n$. 
Then there is a \cqa\ $V\subset X$ of dimension $\le n-1$
such that $\mu$ is smooth in $X\setminus V$. 
\end{thm}

\begin{proof}
Note that the case $n=0$ is trivial. 

We may assume that $\mu=\ff_*\tau$, where $\ff:\U\to X$ is a holomorphic map, $\U\subset\C^N$ is open,  
and $\tau$ is an elementary current of the form \eqref{1elem}
with compact support 
$K\subset \mathcal U$. 
For each multi-index $I=\{i_1,\ldots, i_k\}\subset \{1,\ldots, N\}$,
let 
\[
E_I=\{t_{i_1}=\cdots = t_{i_k}=0\}=E_{i_1}\cap\cdots\cap E_{i_k}, 
\]
where $E_i=\{t_i=0\}$. Moreover, let 
\[
E_I'=\{y\in \U;\  \rank_y \ff|_{E_I}<n\},
\]
where $\ff|_{E_I}$ denotes the restriction of $\ff$ to $E_I$. 
Notice that  $E_\emptyset'=\{y\in \U, \rank_y \ff <n\}$.
Let $E'=\cup_I E_I'$ and let 
$V=\ff(E'\cap K)$. Then $V$ is a \cqa\ in view of Remark ~\ref{vari} and
$\dim V=\rank \ff|_{E'}\le n-1$.

We claim that the restriction to $X\setminus V$ of $\mu$ is smooth.
Let $\chi$ be any smooth cutoff function with support
in $X\setminus V$. We have to prove that  
$\chi\mu$ is smooth.
To this end,  consider $y\in
\ff^{-1}(\supp \chi)\cap K$. Let $I_y=\{i, y\in E_i\}$, i.e., $I_y$ is the
maximal $I$, under inclusion, such that $y\in E_I$. Then there is a
neighborhood $\V_y$ such that $\V_y\cap E_i=\emptyset$ for all
$i\notin I_y$. If $I_y=\{i_1,\ldots, i_k\}$, it follows,  possibly after reordering the
variables,  that $\tau$ is of the form 
\begin{equation*}
\tau=\beta\w\Big[\frac{1}{t_{i_1}^{m_{i_1}}}\Big]\cdots\Big[\frac{1}{t_{i_\ell}^{m_{i_\ell}}}\Big]\dbar\Big[\frac{1}{t_{i_{\ell+1}}^{m_{i_{\ell+1}}}}\Big]\w\ldots
\w \dbar\Big[\frac{1}{t_{i_k^{m_{i_k}}}}\Big],
\end{equation*}
where $\beta$ is smooth in $\V_y$. 

Since $y\notin E'$, possibly after shrinking $\V_y$ we can assume that $\V_y\cap
E'=\emptyset$, which, in particular, implies that $\ff|_{E_{I_y}}$ has rank $n$ in
$E_{I_y}\cap\V_y$. It follows that 
\[ 
d\ff_1\wedge \cdots \wedge d\ff_n \wedge 
dt_{i_1}\wedge \cdots \wedge dt_{i_k}\neq 0
\]
in $E_{I_y}\cap\V_y$ if $\ff=(\ff_1,\ldots, \ff_n)$. 
By the inverse function theorem,  
after possibly shrinking $\V_y$ further, we can thus choose a coordinate system in $\V_y$ so that 
$\ff_1,
\ldots, \ff_n, t_{i_1},\ldots, t_{i_k}$ are the first $n+k$
coordinates. Let $\sigma_1,\ldots, \sigma_{N-n-k}$ be a choice of complementary
coordinate functions.  Then 
\[
\ff: (\ff_1,\ldots, \ff_n, t_{i_1}\ldots t_{i_k}, \sigma_1, \ldots,
\sigma_{N-n-k})\mapsto 
(\ff_1,\ldots, \ff_n),
\]
i.e., $\ff$ is just the projection onto the first $n$ coordinates. 

Let $\chi_y$ be a smooth cutoff function that is $1$ in a
neighborhood of $y$ and has compact support in $\V_y$. Then 
\[
\ff_*(\chi_y \tau) =
\int_{t_i, \sigma_j} 
\chi_y\beta\w\Big[\frac{1}{t_{i_1}^{m_{i_1}}}\Big]\cdots\Big[\frac{1}{t_{i_\ell}^{m_{i_\ell}}}\Big]\dbar\Big[\frac{1}{t_{i_{\ell+1}}^{m_{i_{\ell+1}}}}\Big]\w\ldots
\w \dbar\Big[\frac{1}{t_{i_k^{m_{i_k}}}}\Big], 
\]
which is smooth. 

Since $\ff^{-1}(\supp \chi)\cap K$ is compact, there are finitely many $y$ and
$\V_y$ as above, such that $\cup \V_y$ is a neighborhood of
$\ff^{-1}(\supp \chi)\cap K$. It follows that there is a finite number of
smooth cutoff functions
$\chi_y$ with compact support in $\V_y$ such that $\{\chi_y\}$ is a
partition of unity on $\ff^{-1}(\supp \chi) \cap K$. 
Thus 
\[
\chi\mu = \ff_*(\ff^*\chi \tau) = \sum \ff_* (\chi_y \ff^* \chi \tau)
\]
is smooth, since each term in the rightmost expression is.
\end{proof}

From Theorem ~\ref{tom} and Proposition ~\ref{dimprin} (ii) we get

\begin{cor}
If $\mu\in\W$  vanishes where it is smooth, then $\mu$
vanishes identically.
\end{cor}

\section{Regularity properties of  currents in $\PM_X^Z$ and $\W_X^Z$}\label{taknock}
 

Our first result is a local description of $\PM_X^Z$ when $Z$
is smooth. 
 
\begin{prop}\label{struktur}
Let $\mu$ be a pseudomeromorphic current on a manifold $X$. 
Assume that $\mu$ has support on a submanifold $Z\subset X$ of
codimension $p$. If we choose local coordinates $z_1\ldots, z_{n-p},
w_1\ldots, w_{p}$ in $\U\subset\subset X$ so that $Z=\{w_1=\cdots= w_p=0\}$, 
then, in $\U$,  $\mu$ has a unique finite expansion
\begin{equation}\label{grasvart}
\mu=\sum_r\sum'_{|I|=r}\sum_{m\in\N^p} \mu_{I,m}(z)\otimes
\dbar\frac{1}{w_p^{m_p+1}}\wedge\cdots\wedge\dbar\frac{1}{w_1^{m_1+1}}\wedge
dw_{I_1}\wedge\cdots\wedge dw_{I_r},
\end{equation}
where $\mu_{I,m}$ are pseudomeromorphic currents on $Z$.

\smallskip
\noindent Moreover, $\dbar\mu=0$ if and only if $\dbar\mu_{I,m}=0$ for each $I,m$, and
$\mu$ is in $\W^Z_X$ if and only if $\mu_{I,m}$ is in $\W_Z$ for each $I,m$.
\end{prop}

Notice that the right hand side of \eqref{grasvart} indeed defines a current 
$\mu$ in $\PM_X^Z$  if
$\mu_{I,m}$ are in $\PM_Z$ in view of Lemma~\ref{1tensor}.

For the proof we need the following simple lemma.

\begin{lma}\label{snokus} Let $X$ and $Y$ be analytic spaces and let $p$ be a point in $Y$.

\noindent (i) \  If $\pi\colon X\times  Y\to X$ is the natural projection and $\mu\in \W_{X\times Y}^{X\times\{p\}}$,
then $\pi_*\mu\in \W_X$.

\smallskip
\noindent(ii) \ 
If $\mu$ is in $\W_X$ and $\nu\in\PM_Y$ has support at $p$, 
then $\mu\otimes\nu$ is in $\W_{X\times Y}^{X\times\{p\}}$.
\end{lma}

\begin{proof} 
Let $W$ be a subvarity of $\U\subset X$ of positive codimension. If $\mu$ has support and the SEP on $X\times\{p\}$,
then 
\[\1_{\pi^{-1}W}\mu=\1_{W\times
  Y}\1_{X\times{\{p\}}}\mu=\1_{W\times\{p\}}\mu =0,\]
cf.\ \eqref{skolgard}. 
Thus 
$\1_W \pi_*\mu=\pi_*(\1_{\pi^{-1}W}\mu)=0$,  and so part (i) follows. 
Part (ii) follows from \eqref{pelargonia2}.  %
In fact, assume that the hypothesis is fulfilled. If $W\subset \U\cap X$ has positive codimension,
then $\1_{W\times\{p\}}\mu\otimes\nu=\1_W\mu\times \1_{\{p\}}\nu=0$,
since $\1_W\mu=0$. 
\end{proof}

\begin{proof}[Proof of Proposition~\ref{struktur}]
In view of \cite[Theorem~3.5]{AW3}  it suffices to consider the case
where the terms in \eqref{grasvart} vanishes except for $r=p$,
i.e., $I=(1,\ldots, p)$. Therefore, let us assume from now on that
this is the case. 
Let $\mu_m=\mu_{I,m}$,
$dw=dw_1\wedge\cdots\wedge dw_p$, 
$w^m=w_1^{m_1}\cdots w_p^{m_p}$,
and 
$$
\dbar\frac{1}{w^{m+\1}}=\dbar\frac{1}{w_p^{m_p+1}}\wedge\cdots\wedge\dbar\frac{1}{w_1^{m_1+1}}.
$$

It is readily checked that if $\phi_\ell(z)$ are test forms on
$Z\cap\U$, then 
\begin{equation}\label{potatis}
\int_{z,w}\big (\phi_\ell(z)\otimes w^\ell\big )\w \big
(\mu_m(z)\otimes\dbar\frac{1}{w^{m+\1}}\w dw\big )=\delta_{\ell,m}(2\pi i)^p\int_z \phi_\ell(z)\w \mu_m(z),
\end{equation}
where $\delta_{\ell,m}$ is the Kronecker symbol. 
Let $\pi:\C^n\to \C^{n-p}$ be the projection 
$$
(z_1,\ldots,
z_{n-p},w_1,\ldots, w_{p})\mapsto (z_1,\ldots, z_{n-p}).
$$ 
As a consequence of \eqref{potatis} we have that if $\mu$ has a
representation \eqref{grasvart}, then 
\begin{equation}\label{stensota} 
\pi_*(w^m\mu)=(2\pi i)^p\mu_m.
\end{equation}
Thus the representation \eqref{grasvart} of $\mu$ is unique if it exists.

Now assume that $\mu$ is given, and let
\[
T=\frac{1}{(2\pi i)^p}\sum_{m\in\N^p}
\mu_m(z)\otimes 
\dbar\frac{1}{w^{m+\1}}\wedge
dw,
\]
where $\mu_m$ are defined by \eqref{stensota}. Since $\mu$ has locally finite order this sum is finite and thus defines an element in $\PM_X^Z$. 
We claim that 
\begin{equation}\label{saga}
\mu=T. 
\end{equation}
To prove \eqref{saga}, first notice that for each $j$, $dw_j\w \mu
=dw_j \w T=0$  for degree reasons and  $d\bar w_j\w \mu
=d\bar w_j \w T=0$ by\eqref{polis}, so we only have to check the equality for test forms $\phi$ with no differentials with respect to $w$. A Taylor expansion
with respect to $w$ of such a form $\phi$ gives that
$$
\phi=\sum_{|\ell|<M} \phi_\ell(z)\otimes w^\ell + \Ok(\bar w) + \Ok(|w|^M),
$$
where $\Ok(\bar w) $ denotes terms with some factor $\bar w_j$
and $M$ is chosen so large that $\Ok(|w|^M)\mu=\Ok(|w|^M)T=0$ in $\U$.
Since $\bar w_j \mu=\bar w_j T=0$, cf.,  \eqref{polis}, 
it follows that we just have to check
\eqref{saga} for test forms like
$\phi=\phi_\ell(z) \otimes w^\ell$.
However, it follows immediately from \eqref{potatis} and \eqref{stensota}
that $\mu.\phi=T.\phi$ for such $\phi$, which proves \eqref{saga} and
the first part of the proposition.

Since $\dbar(1/w^{m+\1})\w dw$ is $\dbar$-closed it follows 
by the uniqueness that $\dbar\mu_{m}=0$ for all $m$ 
if (and only if) $\dbar\mu=0$.  
The last statement follows from Lemma~\ref{snokus} $(ii)$.
\end{proof}

This gives us the following extension of Theorem ~\ref{tom}. 

\begin{cor}\label{nusa}
Assume that $\mu$ is a pseudomeromorphic current in $X$ with compact
support in $\U\cap Z$, where $\U$ and $Z$ are as in 
Propostion~\ref{struktur}. Then there is a cqa $V\subset Z\cap \U$ 
of codimension $\ge p+1$ in $\U$ and such that 
\[
\mu=\alpha\wedge\tilde\mu,  
\]
in $\U\setminus V$, where $\alpha$ is a smooth form in $X\setminus V$
and $\tilde\mu$ is a pseudomeromorphic current of bidegree $(0,p)$
with compact support in $Z\cap \U$. 
\end{cor}
 
\begin{proof}
Note that the case $\dim X=0$ is trivial. 

Consider the representation \eqref{grasvart} of $\mu$. As in the proof of
Proposition~\ref{struktur} it suffices to consider terms in the
representation \eqref{grasvart} where $r=p$; let us use the
notation from that proof. 
Choose $M\in N^p$ such that $M_j\geq m_j$ for all  $j$ in
\eqref{grasvart}. Let 
\[
\tilde\mu= \dbar\frac{1}{w^{M+1}}\w dw
\]
and let 
\[
\alpha=
\sum_{m\in\N^p} 
w^{M-m} \mu_{m}(z).
\]
Then clearly $\mu=\alpha\wedge\tilde\mu$ in $\U$.

Since $\mu$ has compact support in $\U\cap Z$, each $\mu_{m}$ has compact
support in $\U\cap Z$ and thus by Theorem ~\ref{tom} there are \cqa s $V_{m}\subset
Z$ of strictly positive codimension, such that $\mu_{m}$ is smooth
outside $V_{m}$. %
Now %
$\alpha$ is smooth in $\U\setminus V\times \C^p_w$, where  
$V:=\cup V_{m}$ is a \cqa\ of codimension $\ge p+1$ in $X$.
Multiplying $\tilde\mu$ by a suitable cutoff function in $\U$ and
replacing $\alpha$ by a smooth form on $X\setminus V$ that coincides
with $\alpha$ on the support of $\mu$, 
 we get the desired
representation of $\mu$ in $\U\setminus V$.  
\end{proof}

The main result in this section is the following local characterization of elements in $\W_X^Z$
in terms of elementary currents.

\begin{thm}\label{tacon} 
Assume that $\mu$ is a \pmm current on $X$ with support on the subvariety $Z$ of
dimension $d$. Then
$\mu\in\W_X^Z$ if and only if there is a locally finite representation
\begin{equation}\label{planta}
\mu=\sum_\ell \ff_* \tau_\ell,
\end{equation}
where $\ff$ is a holomorphic mapping, such that, for each $\ell$, the elementary 
support of $\tau_\ell$ is contained in $\ff^{-1}Z$,
and
the restriction $\tilde\ff_\ell$ of $\ff$ to  the elementary support of  $\tau_\ell$
has generic rank $d$.
\end{thm}

For the proof we need the following lemmas.
\begin{lma}\label{golpe} 
Assume that $\mu=\ff_*\tau$, where $\ff:\U\to X$ and $\tau$ is an
elementary current on $\U$ with elementary support $H$. 
Moreover,  assume that the restriction of $\ff$ to $H$ has generic rank
$d$. Let $W\subset X$ be a subvariety of dimension $\leq d-1$. Then
$\mathbf 1_W\mu=0$. 
\end{lma}

\begin{proof}
In view of Remark ~\ref{gryning} we may assume that $H$ is an irreducible
subvariety of $\U$. 
Assume that $\ff^{-1} W\cap H=H$. Then, since $\ff|_H$ has
generic rank $d$, by the constant rank theorem, there is an open subset
$\W$ of $H$ such that $\ff (\W)$ is a manifold of dimension $d$. 
It follows that $W\supset \ff(\W)$ has dimension $\geq d$,
which contradicts that $W$ has dimension $\leq d-1$. 
Since $H$ is irreducible, we conclude that $\ff^{-1} W\cap H$ is a
subvariety of $H$ of positive codimension. Since $\tau$ has the SEP on
$H$, cf.\ Example ~\ref{elex}, it follows that 
$\mathbf 1_{W}\mu =\ff_*\big (\mathbf 1_{\ff^{-1}W\cap H}
  \tau\big )=0.$
\end{proof}

The next lemma is a generalization of Proposition~\ref{dimprin} (ii). 

\begin{lma}\label{chaflan}
If $\mu\in \W_X^Z$ has support on a \cqa\ $V\subset Z$ of positive
codimension, then $\mu=0$. 
\end{lma}

\begin{proof}
Let $d$ be the dimension of $Z$. By Lemma ~\ref{caput2} there
is a \cqa\ $V'\subset V$ of dimension $\leq d-2$ such that
locally $V\setminus V'$ is contained in a variety of dimension $\leq
d-1$. 
Since $\mu$ has the SEP on $Z$ it
follows that $\supp \mu\subset V'$. By repeating this argument we get
that $\mu$ vanishes, cf.\ the proof of Proposition ~\ref{dimprin}. 
\end{proof}

\begin{proof}[Proof of Theorem~\ref{tacon}]
Let $\ff\colon \U\to X$ be a holomorphic mapping and let
$\tau$ be elementary with compact support in $\U$. Moreover assume
that the restriction 
of $\ff$ to the elementary support $H$ of $\tau$ has generic rank
$d$ and that $H\subset\ff^{-1}Z$. 
Then clearly $\ff_*\tau$ has support on $Z$. 
Let $\V$ be an open subset of $X$ and let $W\subset \V\cap Z$ be a
subvariety of positive codimension. We claim that $\mathbf 1_W\mu=0$. 
To prove this it suffices to show that $\mathbf 1_W\chi\mu=0$ for each
smooth cutoff function $\chi$ with compact support in $\V$. 
This however follows from Lemma ~\ref{golpe} applied to $\hat \ff_*
(\ff^*\chi\tau)$, where $\hat\ff:\ff^{-1}\V\to \V$ is the restriction of
$\ff$ to $\ff^{-1}\V$. 
Hence $\ff_*\tau$ is in $\W_X^Z$ and thus the ``if''-part of the proposition is proved.
\smallskip

For the converse assume that $\mu$ is in $\W_X^Z$.
With no loss of generality we can assume that  $\mu$ has compact support, so that we
have a finite representation like \eqref{planta}, without any special
assumption on the $\ff$ and $\tau_\ell$. In view of Lemma ~\ref{batong}
(and its proof)
we  may also assume that all the elementary supports of the
$\tau_\ell$ are contained in  $\ff^{-1} Z$.  
Consider $\tau_\ell$ such that $\tilde\ff_\ell$ has generic rank $\geq
d+1$. Since $\ff_*\tau_\ell$ is contained in $Z$ of dimension $d$, $\ff_*\tau_\ell$
vanishes by Lemma ~\ref{golpe}. Thus we may assume from now on that
$\rank \tilde\ff_\ell \leq d$ for all $\ell$. Now write $\mu=\mu'+\mu''$, where
$\mu'$ is the sum of all $\ff_*\tau_\ell$ for which $\rank\tilde\ff_\ell=d$. Then
$\mu'$ is in $\W_X^Z$ by the first part of the proof. Hence so is $\mu''$.  

If $\rank \tilde\ff_\ell\le d-1$, then  $\supp\ff_*\tau_\ell$ is contained in a \cqa\ $\ff(H)$ of
dimension $\leq d-1$. Thus $\supp \mu''$ is as well, and hence
it vanishes in view of Lemma
~\ref{chaflan}. Hence $\mu=\mu'$.  
 \end{proof}


As an immediate consequence we get: 

\begin{cor} If $\mu_j$ is in $\W_{X_j}^{Z_j}$, $j=1,2$, then 
$\mu_1\otimes \mu_2$ is in
$\W_{X_1\times X_2}^{Z_1\times Z_2}$.
\end{cor}

\section{Stalkwise injectivity of $\PM_X$}\label{stalk}
In this section $X$ is a smooth manifold. 

\begin{thm} \label{statut}
Let $X$ be a smooth manifold. The sheaves $\PM_X$ are stalkwise injective
for all $\ell,k$. 
\end{thm}

The corresponding statement for general currents $\mathcal C_X$ is a classical result
due to Malgrange. For a quite simple proof by integral formulas,
see \cite[Section 2]{Aext}.   

\smallskip
\noindent Theorem ~\ref{statut} means:   {\it If $E_k$ are holomorphic vector bundles and
\begin{equation} \label{billing}
\cdots \stackrel{f_2}{\to}\Ok(E_1)\stackrel{f_1}{\to} \Ok(E_0)\to \F\to 0
\end{equation}
is a locally free resolution of a coherent sheaf $\F$ over $X$,  
then the induced sheaf complex
\begin{equation*}
0\to\Homs_\Ok(\F,\PM)\to \Homs_\Ok(\Ok(E_0),\PM)\stackrel{f_1^*}{\to}
\Homs_\Ok(\Ok(E_1),\PM)\stackrel{f_2^*}{\to}\cdots
\end{equation*}
is exact.}

\smallskip
\noindent The exactness at the first two places is trivial, so we are to prove that
the equation $f_{k}^*u=\mu$ kan be (locally) solved in $\Homs_\Ok(\Ok(E_{k-1}),\PM)$ for each
$\mu$ in $\Homs_\Ok(\Ok(E_{k}),\PM)$ such that $f_{k+1}^*\mu=0$, $k=1,2, \ldots$.  

Note that Theorem ~\ref{statut} is equivalent to that $\PM_X^{\ell,k}$
is injective for each $\ell, k$.

\begin{ex} \label{putin}
Let $f$ be a single generically non-vanishing holomorphic function.
Then 
$$
0\to \Ok \stackrel{f}{\to}\Ok\to \Ok/(f)\to 0
$$
is a free resolution of $\F=\Ok/(f)$.  The condition $f_2^*\mu=0$ is
vacuous in this case so the stalkwise injectivity means that 
the equation $f\nu=\mu$ is locally solvable for any \pmm $\mu$, which is
precisely the content of \cite[Proposition 3.1]{AW3} (in case $X$ is smooth).
\end{ex}

We postpone the proof of Theorem ~\ref{statut} to Section ~\ref{bevis} and first
discuss some consequences. To this end we need some facts about residues as well as
solvability of the $\dbar$-equation for \pmm currents.

\subsection{Residues associated to a locally free resolution}\label{pyton}
Consider a locally free resolution \eqref{billing} of the coherent sheaf
$\F$ on $X$,
let $E=\oplus E_k$ and $f=f_1+f_2+\cdots$. We equip $E$ with a 
superstructure so that $E_+=\oplus E_{2j}$ and $E_-=\oplus E_{2j+1}$.
Then both $f$ and $\dbar$ are odd mappings on the sheaf
$\Cu(E)$ of $E$-valued currents,  and thus so is $\nabla=f-\dbar$.
Let $\nabla_{\End E}$ be the induced mapping on endomorphisms
on $E$, see \cite{AW1} for more details.  

Let us choose Hermitian metrics on the vector bundles $E_k$, and let $U$ and $R$ be the associated $\End E$-valued
principal value, and residue currents, respectively, as defined in  
\cite[Section 2]{AW1}, so that $\nabla_{\End E} U=I_E-R$.
It follows from the construction that $U$ is almost semi-meromorphic
on $X$ and that $R$ is the residue of $U$, cf.,  Section ~\ref{lisa2}. 
Thus $R$ has support on $Z:=ZSS(U)$, which by construction is precisely the analytic set where 
$\F$ is not locally free, or equivalently,
the set where the complex $\big (\Ok(E_\bullet),f_\bullet\big )$ is not
pointwise exact.  

If $\chi\sim\chi_{[1,\infty)}$, as before, and   
$\chi_\epsilon=\chi(|h|^2/\epsilon)$, where $h$ is a holomorphic tuple
whose common zero set is $Z$, then
$U_\epsilon:=\chi_\epsilon U$ is smooth for $\epsilon>0$ and 
$U_\epsilon\to U$ when $\epsilon\to 0$; in fact,
  $\chi(|f_1|^2/\epsilon)$ will do.  
We can define the smooth form $R_\epsilon$ so that
$\nabla_{\End E} U_\epsilon=I_E-R_\epsilon$.    
Then clearly
$R_\epsilon\to R$ when $\epsilon\to 0$. 
Since $\nabla_{\End E} U=I_E$ outside $Z$ 
it follows that 
\begin{equation}\label{reps}
R_\epsilon=(1-\chi_\epsilon)I_E+\dbar\chi_\epsilon\w U.
\end{equation}

Let $U^\ell_k$ and $R_k^\ell$ be the components of $U$ and $R$,
respectively,  that take values in $\Hom(E_\ell, E_k)$. 
By \cite[Theorem 3.1]{AW1},  $R^\ell_k=0$ when $\ell\ge 1$.  Thus
we can write $R_k$ rather then $R^0_k$.

\begin{ex} \label{putin2} For the resolution in Example~\ref{putin}, 
we have $U=1/f$ and $R=\dbar(1/f)$.
\end{ex}

 \subsection{The $\dbar$-equation for \pmm currents} \label{buss}
Let us recall how one can solve the $\dbar$-equation by means of simple
integral  formulas. 
From  Example ~\ref{bex2} we know that  
$$
B':=\sum_{k=1}^n b'\w (\dbar b')^{k-1}
$$
is almost semi-meromorphic in $\C^n_\eta$ if
$
b'=\partial|\eta|^2/
2\pi i|\eta|^2.
$

Thus $B'\otimes 1$ is almost semi-meromorphic in $\C_\eta^n\times \C_\xi^n$, 
and by a linear change of coordinates we find that 
$B:=\eta^* B'$ is almost semi-meromorphic in $\C_\zeta^n\times \C_z^n$,
if $\eta(\zeta,z)=\zeta-z$.  If $\mu$ is any current with compact support
in $\C^n_\zeta$, one can define the convolution operator
\begin{equation*}
\K\mu(z)=\int_\zeta B_{n,n-1}(\zeta,z)\w\mu(\zeta),
\end{equation*}
where $B_{n,n-1}$ denotes the component of bidegree $(n,n-1)$, 
for instance by replacing $B$ by the regularization
$B_\epsilon=\chi(|\zeta-z|^2/\epsilon) B$ and taking the limit
when $\epsilon\to 0$.
More formally,  $\K\mu = p_*(B_{n,n-1}\w \mu\otimes 1)$, where
$p$ is the natural projection $(\zeta,z)\mapsto z$.
If $\mu$ is pseudomeromorphic, then also $\mu\otimes 1$ is, cf.\ Lemma~\ref{1tensor}, and thus
$B\w \mu\otimes 1$ is just multiplication by the almost semi-meromorphic
current $B$, see Theorem~\ref{hittills}. 
It follows that $\K\mu$ is \pmm if $\mu$ is.

The top degree term $B_{n,n-1}$ is the classical Bochner-Martinelli
kernel. 
The other terms in $B$ will play an important role below. 
It is well-known that 
\begin{equation}\label{bkopp}
\mu=\dbar \K \mu+ \K\dbar\mu.
\end{equation}

\begin{prop}\label{reso1}
If $X$ is a smooth manifold, then 
\begin{equation*}
0\to \Omega^p_X\to \PM_X^{p,0}\stackrel{\dbar}{\to}\PM_X^{p,1}\stackrel{\dbar}{\to}\cdots
\end{equation*}
is a fine resolution of $\Omega^p_X$.  
\end{prop}

Here $\Omega_X^p$ denotes the sheaf of holomorphic $p$-forms.
This proposition is implicitly proved in \cite{AS} but for the reader's convenience
we supply a simple direct argument.

\begin{proof}
Since the case $k=0$ is well-known let us assume that $\mu$ is \pmm of bidegree $(p,k)$, $k\ge 1$,
and $\dbar \mu=0$. Fix a point $x\in X$ and let  $\chi$ be a cutoff function
in a coordinate \nbh of $x$ that is identically $1$ in a \nbh of $x$.  We can then apply 
\eqref{bkopp} to $\chi\mu$ and so we get that $\chi\mu=\dbar \K(\chi\mu)+\K(\dbar\chi\w\mu)$.
Now $\K(\chi\mu)$ is \pmm in view of Proposition ~\ref{pelargonia3} below. Furthermore, $\K(\dbar\chi\w\mu)$ is
smooth where $\chi=1$ since $B$ only has singularities at the diagonal. 
Since this term in addition is $\dbar$-closed near  $x$
it is locally of the form $\dbar\psi$ for some smooth $\psi$.
 It follows that there is a local \pmm solution at $x$ to $\dbar \nu=\mu$.   
\end{proof}

\begin{prop}\label{pelargonia3} 
The integral operator $\K$ maps \pmm currents on $\C^n$ with compact support 
into $\W(\C^n)\subset\PM(\C^n)$. 
\end{prop}

This is an immediate consequence of

\begin{prop}\label{pelargonia4}
If $A$ is almost semi-meromorphic on $X\times Y$, 
$\mu\in\PM(X)$ has compact support, and    $\pi\colon X\times Y\to Y$
is the natural projection,
then $\pi_*(A\w \mu\otimes 1)$ is in $\W(Y)$.
\end{prop}

\begin{proof} 
By Theorem ~\ref{hittills} $\pi_*(A\w \mu\otimes 1)$ is in
$\W(Y)$. Assume that 
$V\subset\U\subset Y$ has positive codimension. Then, in view of \eqref{brutus2}, \eqref{pelargonia}  and \eqref{pelargonia2}, we have 
\begin{multline*}
\1_V\pi_*\big (A\w (\mu\otimes 1)\big)=\pi_*\big(\1_{ X\times V}(A\w (\mu\otimes 1))\big)=\\
\pi_*\big(A\w\1_{X\times V}(\mu\otimes 1)\big)=
\pi_*\big(A\w (\1_{X}\mu\otimes\1_V 1)\big)=0,
\end{multline*}
since $\1_V 1=0$. 
\end{proof}

\subsection{A generalization of the Dickenstein-Sessa decomposition}\label{dickenstein}
Let $Z$ be a reduced analytic variety of pure codimension $\nu$. 
A $(p,\nu)$-current $\mu$ on $X$ is a \emph{Coleff-Herrera current} on $Z$, $\mu \in
\CH_{p}^{Z}$,
if $\dbar \mu = 0$, $\bar{\psi} \mu = 0$ for all holomorphic functions $\psi$ vanishing
on $Z$, and $\mu$ has the SEP with respect to $Z$; see, e.g., \cite[Section~6.2]{Bj}. 
Let $(\Cu^Z_{p,k}, \dbar)$ be the Dolbeault complex of $(p,*)$-currents on $X$
with support on $Z$. 
Dickenstein and Sessa proved in \cite{DS, DS2}\footnote{In \cite{DS}
  the Dickenstein-Sessa decomposition \eqref{dsa} was proved for complete
  intersections $Z$ and in \cite[Proposition~5.2]{DS2} for arbitrary
  $Z$ of pure dimension.}, see also \cite{Aext, Bj}, 
that Coleff-Herrera currents are canonical
representatives in moderate cohomology, i.e., 
\begin{equation}\label{dsa}
\mathcal Ker_{\dbar}\,\Cu^Z_{p,\nu}=\CH^Z_{p}\oplus \dbar
\Cu^Z_{p,\nu-1}; 
\end{equation}
in other words, each $\dbar$-closed current $\mu$ with support on $Z$ has a unique 
decomposition
\begin{equation}\label{spott}
\mu=\mu_1+\dbar\gamma,
\end{equation}
where $\mu_1$ is in $\CH^Z_{p}$ and $\gamma$ has support on $Z$.
\smallskip

Let $\F$ be a coherent sheaf over $X$ and let
\eqref{billing} be a locally free resolution.  
Combining Theorem~\ref{statut} and Proposition~\ref{reso1} we find that 
$$
\M_{\ell,k}=\Homs_\Ok\big (\Ok(E_\ell,\PM^{p,k})\big )
 $$
is a double complex with vanishing cohomology except at $\ell=0$ and $k=0$, where
the kernels are $\Homs_\Ok(\F,\PM^{p, k})$ and $\Homs_\Ok(\Ok(E_\ell),\Omega^p)$,
respectively.  
The same holds if $\PM^{p,\bullet}$ are replaced by 
the sheaves of general currents $\Cu^{p,\bullet}$, in view of 
the well-known local solvability of $\dbar$ for 
$\Cu$, and Malgrange's theorem.  
By standard cohomological algebra
we get 

\begin{thm} If $\F$ is a coherent sheaf over $X$ and 
\eqref{billing} is a locally free resolution, then there are canonical isomorphisms
\begin{multline}\label{valand}
\Exts_\Ok^k(\F,\Omega^p)\simeq \Hs^k(\Homs_\Ok(\Ok(E_\ell),\Omega^p), f^*_\bullet)\simeq \\
\Hs^k(\Homs_\Ok(\F,\PM^{p,\bullet}),\dbar)\simeq
\Hs^k(\Homs_\Ok(\F,\Cu^{p,\bullet}),\dbar), \quad k\ge1.
\end{multline}
\end{thm}
The novelty in \eqref{valand} is the representation of 
$\Exts_\Ok^k(\F,\Omega^p)$
by Dolbeault cohomology for the smaller sheaves of currents $\PM$.
In particular we have the decompositions
\begin{equation}\label{valand2}
\Kers_{\dbar}\Homs_\Ok(\F,\Cu^{p,k})=
\Hs^k(\Homs_\Ok(\F,\PM^{p,\bullet}),\dbar)\oplus
\dbar \Homs_\Ok(\F,\Cu^{p,k-1}).
\end{equation}
That is, each $\dbar$-closed $\mu$ in $\Homs(\F,\PM^{p,k})$ 
has a decomposition \eqref{spott} 
where $\mu_1$ is determined modulo
$\dbar\Homs(\F,\PM^{p,k-1})$ 
and $\gamma$ is in $\Homs(\F, \Cu^{p,k-1})$.

\begin{remark}
From \cite[Theorem 7.1]{Aext}, see also  
\cite[Remark 4]{Aext}, it follows that the second mapping in \eqref{valand}
is realized by 
\begin{equation*}
\xi\mapsto \xi\cdot R_k,
\end{equation*}
for  $\xi$  in $ \Homs_\Ok(\Ok(E_k),\Omega^p)$
such that  $f^*_{k+1}\xi=0$. 
\end{remark}

Let us now assume that $\F=\Ok/\J$, where $\J$ is an ideal sheaf of
pure codimension $\nu$, and let $Z$ be the associated zero set.
It is not too hard to see that $\CH_{p}^{Z}$
is precisely the sheaf of $\dbar$-closed currents in $\PM^Z_{p,\nu}$, see
e.g., \cite{Aext}.  
Taking $k=\nu$ in the last equality in \eqref{valand} we get, in view of the dimension principle, that 
$$
\Hs^k(\Homs_\Ok(\Ok/\J,\PM^{p,\bullet}),\dbar)=
\Kers_{\dbar}\Homs_\Ok(\Ok/\J,\PM^{p,\nu})=\Homs_\Ok(\Ok/\J,\CH^Z_{p}),
$$
cf., e.g., \cite[Theorem 1.5]{Aext} and ~\cite{Bj}.  Notice that  $\Homs_\Ok(\Ok/\J,\CH^Z_{p})$ is the sheaf
of Coleff-Herrera currents $\mu$ such that $\J \mu=0$. 

Let $\I\subset \Ok$ be the radical ideal associated with $Z$, i.e., the sheaf of 
functions that vanish on $Z$.  If $\mu$ is any current of bidegree $(p,\nu)$
with support on $Z$, i.e., in $\Cu^Z_{p,\nu}$,  
then locally $\J\mu=0$ if $\J=\I^m$ for sufficiently large $m$. 
Applying \eqref{valand2} to $\J=\I^m$ for $m=1,2,\ldots$, and $k=\nu$,
we get the
Dickenstein-Sessa decomposition \eqref{dsa}.

Notice that $\Homs_\Ok(\Ok/\J,\PM^{p,k})$ is the subsheaf of
$\mu$ in $\PM^{p.k}$ such that $\J\mu=0$. In particular such $\mu$
must have support on $Z$. Arguing as in the case $k=\nu$  above
we get from \eqref{valand2}
the following extension of \eqref{dsa} for general $k$. 
\begin{cor}[Generalized Dickenstein-Sessa decomposition]\label{gdsd}
 If $\mu$ is a $\dbar$-closed $(p,k)$-current with support on $Z$, then there is a decomposition
\eqref{spott}, where $\mu_1$ is in $\mathcal Ker_{\dbar}\, \PM^Z_{p,k}$,
determined modulo  $\dbar \PM^Z_{p,k-1}$, and
$\gamma$ has support on $Z$. 
\end{cor}
In \cite{S} Samuelsson Kalm proves a generalization of this 
decomposition, where $\PM$ are replaced by certain smaller 
subsheaves of $\PM$.

\section{Proof of Theorem ~\ref{statut}}\label{bevis}
We first consider the case when $\F=\Ok/(f)$ as in Examples~\ref{putin}
and ~\ref{putin2}.
We will provide an argument in this special case that admits an extension to 
a proof of Theorem ~\ref{statut}. 
Recall from Example ~\ref{skot} that if $\mu$ is in $\W$,  then 
$f(1/f)\mu=\1_{\{f\neq 0\}}\mu=\mu$.  Notice that also $f \dbar \big((1/f)\mu \big)=
\dbar\big(f(1/f)\mu\big)=\dbar\mu$. 
In view of \eqref{bkopp}  and Proposition~\ref{pelargonia3} (and the
dimension principle if $\mu$ is $(*,0)$), 
an arbitrary \pmm current with compact support can be written 
$$
\mu=\mu_1+\dbar\mu_2,
$$
where $\mu_j$ are in $\W$.  If 
$$
\nu=\frac{1}{f}\mu_1+\dbar\big(\frac{1}{f}\mu_2\big),
$$
thus  $f\nu=\mu$.

\smallskip
We now turn our attention to a general locally free resolution
\eqref{billing}. 
If $k\ge 1$, $\mu\in \Homs_\Ok(\Ok(E_k),\PM)$, and $f_{k+1}^*\mu=0$ in a \nbh of a point $x$, we must then find a \pmm current $\nu\in\Homs_\Ok(\Ok(E_{k-1}),\PM)$
in a \nbh of $x$ such that $f_k^*\nu=\mu$.

With no loss of generality we may assume that $\mu$ has bidegree $(n,*)$,
cf., \cite[Theorem 3.5]{AW3}, 
and compact support in an open ball $\U\subset \C^n$ with center $x$,
and that $f^*\mu=0$.
We will construct integral operators $\Ak$ and $\F$ for such $\mu$
such that 
\begin{equation}\label{sport}
\mu=f^*\Ak\mu +\F\mu
\end{equation}
and  $\Ak\mu$ and $\F\mu$ are pseudomeromorphic. If $\F\mu=0$, then $\nu=\Ak\mu$ thus solves our problem.  This is in fact the case
if the support of $\mu$ is discrete.
In general, unfortunately $\F\mu$ does
not vanish, or at least we cannot prove it. However, we can prove that
$\F\mu$ has substantially "smaller" support than $\mu$, see Lemma ~\ref{formel} below.
In particular, $\supp (\F\mu)\subset \supp \mu$.
Since $f^*\mu=0$, \eqref{sport} implies that $f^*\F\mu=0$. 
Therefore we can apply \eqref{sport} to
$\F\mu$, and then  
$$
\mu=f^*\big(\Ak\mu+\Ak \F\mu\big) + \F^2\mu.
$$
Again $f^*\F^2\mu =0$ so we can iterate and in view of 
Lemma~\ref{borg} below we
obtain a solution $\nu=\Ak(\mu+ \F\mu + \F^2\mu + \cdots)$ to $f^*\nu=\mu$ after a finite  number of steps. Thus Theorem ~\ref{statut} follows.
It thus remains to construct integral operators $\Ak$ and $\F$
with the desired properties. 
%

\subsection{The integral operators $\Ak$ and $\F$ in $\U$}
Let us recall some facts from \cite[Section~9]{Aint1} about integral representation in $\U$.   
 Let $F\to \U$ be a holomorphic vector bundle and
assume that 
$g=g_{0,0}+\cdots +g_{n,n}$ is a smooth form in $\U_\zeta\times\U_z$,
where lower indices denote bidegree,  such that 
$g$ takes values in $\Hom(F_\zeta,F_z)$ at the point $(\zeta,z)$.
We will also assume that $g$ has no
holomorphic 
differentials\footnote{We are only interested here in integral formulas for forms of bidegree $(0,*)$ and therefore we can take $d\zeta_j$ instead of $d\eta_j$ in \cite{Aint1}.}
with respect to $z$. Let $\delta_\zeta$ denote interior multiplication with the vector field 
$$
2\pi i\sum_1^n \zeta_j\frac{\partial}{\partial\zeta_j}
$$
and let $\nabla_\zeta=\delta_\zeta-\dbar$.
We say that $g$ is a {\it weight} (with respect to $F$) if
$\nabla_\zeta g=0$ and if in addition $g_{0,0}=I_F$, the identity mapping on $F$,
on the diagonal in $\U\times\U$.

From now on we only consider the components of the form $B$ from Section ~\ref{buss}
above
with no holomorphic differentials with respect to $z$. For simplicity we denote
it by $B$ as well.  Let $g$ be a weight with respect to $F$.  For test forms  $\phi(\zeta)$
of bidegree $(0,*)$ in $\U$ with values in $F$ we have the Koppelman formula 
\begin{equation}\label{koppvikt}
\phi(z)=\dbar \int_ \zeta (g\w B)_{n,n-1}\w\phi + 
\int_\zeta (g\w B)_{n,n-1}\w \dbar \phi+\int_\zeta g_{n,n}\w\phi, \quad z\in\U.
\end{equation}
The case when $F$ is the trivial line bundle, is proved in
\cite[Section~9]{Aint1} and the general case is verified in exactly the same way.

\smallskip 
Consider now our (locally) free resolution \eqref{billing} in $\U$,  
choose Hermitian metrics on the vector bundles $E_k$, and let $U_\epsilon$ and $R_\epsilon$ be the associated currents as in Section~\ref{pyton} above.
Let $H$ be a Hefer morphism with respect to $E$ that is holomorphic in both $\zeta$ and $z$.
See, e.g., \cite[Section 5]{AW1} for the definition
and basic properties of Hefer morphisms; in particular $H$ is an $\End
E$-valued holomorphic form. 
Then
$$
g_\epsilon:=f(z)HU_\epsilon+HU_\epsilon f+ HR_\epsilon
$$
is a smooth weight with respect to $E$. Here $f, U_\epsilon,
R_\epsilon$ stands for $f(\zeta), U_\epsilon(\zeta), R_\epsilon(\zeta)$.    
Let $g_\epsilon^k$ be the component of $g_\epsilon$ that is a weight
with respect to $E_k$.  For test forms $\phi$ 
of bidegree $(0,*)$ with values in $E_k$ we have then, in view of \eqref{koppvikt},  the representation
\begin{equation}\label{appar}
\phi(z)=\dbar \int_\zeta (g^k_\epsilon\w B)_{n,n-1}\w\phi+ \int_\zeta (g^k_\epsilon\w B)_{n,n-1}\w\dbar\phi+
\int_\zeta (g^k_\epsilon)_{n,n}\w\phi.
\end{equation}
By the way, the last term vanishes unless $\phi$ has bidegree $(0,0)$,
since $g^k_\epsilon$ contains no anti-holomorpic differentials with respect
to $z$ so that $(g^k_\epsilon)_{n,n}$ must have bidegree $(n,n)$ with respect to $\zeta$.

Let $R^k$ and $R^k_\epsilon$ 
be the components of $R$ and $R_\epsilon$, respectively, that take values in $\Hom(E_k,E_*)$, and define $U^k$ and $U^k_\epsilon$ analogously.  Let $H^k$ be the component of $H$ that takes values in $\Hom(E_*,E_k)$.  
Then
\begin{equation}\label{pulka}
g^k_\epsilon=  f_{k+1}(z) H^{k+1}U^k_\epsilon +H^{k}U^{k-1}_\epsilon f_{k} + H^k R^k_\epsilon.
\end{equation}

Now assume that $\mu$ is a \pmm  $(n,q)$-current with compact support
in $\U$ and taking  values in $E^*_k$ for $k\geq 1$. Integrating $\mu$ against \eqref{appar} for test forms $\phi$
with values in $E_k$ we get
\begin{equation*}
\mu(\zeta)=\int_z (g^k_\epsilon\w B)^*_{n,n-1}\w\dbar\mu+ \dbar \int_z (g^k_\epsilon\w B)^*_{n,n-1}\w\mu+
\int_z (g^k_\epsilon)^*_{n,n}\w\mu
\end{equation*}
(up to signs). 
Assuming that $f^*_{k+1}\mu=0$ and plugging in \eqref{pulka} we get
\begin{multline}\label{kaka1}
\mu(\zeta)=
f_k^*(\zeta) \int_z (H^kU^{k-1}_\epsilon B)^*_{n,n-1}\w\dbar\mu+
\dbar \Big(f^*_k(\zeta)\w\int_z (H^kU^{k-1}_\epsilon
B)^*_{n,n-1}\w\mu\Big)+ \\
f_k^*(\zeta)\int_z (H^kU^{k-1}_\epsilon)^*_{n,n}\w\mu+\int_z(H^kR^k_\epsilon\w B)_{n,n-1}^*\w\dbar\mu 
\\  
\dbar \int_z(H^kR^k_\epsilon)\w B)_{n,n-1}^*\w\mu + 
\int_z(H^kR^k_\epsilon)^*_{n,n}\w \mu.
 \end{multline}
To simplify notation we now suppress the lower indices, and instead tacitly understand that we only consider products
of terms such that the total bidegrees add up to the desired one. We can then 
write \eqref{kaka1} more suggestively as
\begin{multline}\label{kaka2}
\mu(\zeta)=f_k^*(\zeta) (U^{k-1}_\epsilon)^*(\zeta)\int_z (H^k)^* \w B \w\dbar\mu+ 
\dbar \Big( f^*_k(\zeta)   \w (U^{k-1}_\epsilon)^*(\zeta)\int_z
(H^k)^*\w B   \w\mu\Big)+ \\
f_k^*(\zeta)(U_\epsilon^{k-1})^*(\zeta)\int_z (H^k)^*\w\mu + (R^k_\epsilon)^*(\zeta)  \int_z(H^k)^*\w B \w\dbar\mu+\\
\dbar\Big(  (R^k_\epsilon)^*(\zeta)  \int_z(H^k)^*\w B \w\mu\Big)
+(R^k_\epsilon)^* (\zeta)\int_z(H^k)^*\w \mu.
 \end{multline}
Since $f^*_k$ and $\dbar$ have odd order with respect to the 
superstructure, cf., Section~\ref{pyton},  they anti-commute
and thus we can we can write \eqref{kaka2} as 
\begin{equation*}%
\mu(\zeta)= f_k^*(\zeta)\Ak_\epsilon\mu(\zeta)+\F_\epsilon\mu(\zeta),
\end{equation*}
where
\begin{multline*}
\Ak_\epsilon\mu=  (U^{k-1}_\epsilon)^*(\zeta)\int_z (H^k)^* \w B \w\dbar\mu-\\
\dbar \Big( (U^{k-1}_\epsilon)^*(\zeta)\int_z (H^k)^*\w B   \w\mu\Big)+
 (U_\epsilon^{k-1})^*(\zeta)\int_z (H^k)^*\w\mu
\end{multline*}
and
\begin{multline*}
\F_\epsilon\mu(\zeta)= (R^k_\epsilon)^*(\zeta) \int_z  (H^k)^*\w B
\w\dbar\mu+\\
\dbar\Big( (R^k_\epsilon)^*(\zeta) \int_z  (H^k)^*\w B \w\mu\Big) +  
(R^k_\epsilon)^* (\zeta)\int_z(H^k)^*\w \mu.
\end{multline*}

\begin{lma}  
Each term in $\Ak_\epsilon\mu$ and $\F_\epsilon\mu$ tends to a
pseudomeromorphic current when $\epsilon\to
  0$.  
\end{lma}
We denote the limits of $\Ak_\epsilon\mu$ and $\F_\epsilon\mu$ by
$\Ak\mu$ and $\F\mu$, respectively.

\begin{proof}  In view  Proposition~\ref{pelargonia4},
$$
\gamma:=\int_z (H^k)^* \w B \w\dbar\mu 
$$
is in $\W(\U)$, since $B$ is almost semi-meromorphic and $H$ is smooth.   Since $U$ is almost semi-meromorphic, by  Theorem~\ref{hittills}  we can 
form the \pmm current
$T=(U^{k-1})^*\w\gamma$, which is in $\W(\U)$ in view of Proposition~\ref{koko} (with $Z=X=\U$). 
Since $U^{k-1}_\epsilon=\chi(|h|^2/\epsilon)U^{k-1}$, where
$Z(h)=ZSS(U)=ZSS(U_k)$, cf.\ \cite[Section~2]{AW1}, 
it follows that 
$(U^{k-1}_\epsilon)^*(\zeta)\w \gamma\to T$, cf.\ \eqref{asm4}. Thus the first term in $\Ak_\epsilon\mu$ tends to a \pmm current
in $\U$. 
Moreover, from the definition \eqref{reps} for  $R_\epsilon$ it follows that the limit
of the first term in $\F_\epsilon$ equals $r(U^k)\w
\gamma=R^k\wedge \gamma$, cf., \eqref{dagis}. 

Since $\dbar$ preserves pseudomorphicity, the same argument works
 for the other terms in $\Ak_\epsilon\mu$ and $\F_\epsilon$.  
\end{proof}

Recall that since \eqref{billing} is exact the current $R^k$ vanishes when
$k\ge 1$, cf., Section~\ref{pyton}. 
Unfortunately, from this 
we cannot conclude that
 the limit $\F\mu$ vanishes in general; cf., 
\cite[Example 4.23]{AW3}.
However, as we now shall see, the support of $\F\mu$ is small in the
following sense:

\begin{lma}\label{formel}

\noindent
$(i)$\ 
The support of $\F\mu$ is contained in the support of $\mu$. 
\smallskip 

\noindent $(ii)$\
Assume that $\mu$ has compact support on a submanifold $Z\subset \V$ of codimension $\geq p$, where $\V$ is an open
subset of $\U$. Then there is a 
\cqa\ $V\subset Z$ of codimension $\geq p+1$ such that $\supp
(\F\mu)\subset V$. 
\end{lma}

\begin{proof}
First notice that $(R^k_\epsilon)^*(\zeta)(H^k)^*\wedge \mu$ is a smooth form
times the tensor product of $(R^k_\epsilon)^*$ and $\mu$. It follows
that the last term in 
the definition of $\F_\epsilon\mu$ 
tends to $0$, since 
$R^k=0$. We thus have to deal with the first two terms.

\smallskip 

To prove $(i)$ we note that if $\mu=0$ close to $x\in \U$, then 
$$
\int_z(H^k)^*\wedge B\wedge \dbar \mu
$$ 
is smooth close to $x$, since $B$ is
smooth outside the diagonal in $\U\times \U$. Thus, close to $x$, the
first term in $\F_\epsilon\mu$ tends to $(R^k)^*$ times a
smooth form and thus the limit vanishes since $R^k=0$. The second term in
$\F_\epsilon\mu$ tends to $0$ for the same reason. 

\smallskip 

To prove $(ii)$, let us consider the limit 
\begin{equation}\label{apa3}
T\mu=\lim_{\epsilon\to 0} (R^k_\epsilon)^*(\zeta) \w (H^k)^*\w B\w\mu, 
\end{equation}
where, as before, we use the 
 simplified notation and in fact only take into account terms of  
$
(R^k_\epsilon)^*(\zeta) (H^k)^*\w B
$
of total bidegree $(n,n-1)$.   
Note that $T\mu$ is the product of a residue
of an almost semimeromorphic current $(R^k)^*$ and a pseudomeromorphic current,
cf.\ Definition ~\ref{bounce}, and thus is pseudomeromorphic. 
Let 
$$
\T\mu=\int_z T\mu.  
$$
Then 
$$
\F\mu=\T\dbar\mu +\dbar (\T\mu).
$$
Thus it is enough to prove (ii) for $\T\mu$ instead of $\F\mu$.

\begin{lma} \label{surpuppa} 
Assume that $\mu$ has compact support on a 
subvariety $W\subset \V$ of codimension $p$ and  
\begin{equation}\label{ballong}
\mu=\alpha\wedge\tilde\mu,
\end{equation}
where $\alpha$ is smooth and $\tilde \mu$ has support on $W$ and  bidegree
$(*,p)$. Then $\T\mu=0$. 
\end{lma}

\begin{proof}
Notice, in view of the proof of $(i)$ above, that $(i)$ holds for $\T$ instead of $\F$. 
Therefore suffices to show that $\T\mu=0$ in $\V$. 
Outside the diagonal in $\V\times \V$, the current $B$ is smooth, and hence $T\mu$ vanishes, as it is a smooth form times the tensor product
of  $(R^k)^*$ and $\mu$, and $R^k=0$.  If $\mu$ is of the form \eqref{ballong}, therefore \eqref{apa3} 
is a smooth form
$\alpha$ times a \pmm current with support on $(\V\times W)\cap \Delta$
that is a subvariety of $\V\times\V$ of codimension $\ge n+ p$. On the other
hand the antiholomorphic degree is $n-1+p$. Thus $T\mu$  must
vanish in view of the dimension principle. It follows that $\T\mu$ vanishes. 
\end{proof}


We can now conclude the proof of $(ii)$ for $\T$. 
We can cover $\V$ by finitely many
neighborhoods $\V_j$ such that $\V_j$ and $Z\cap\V_j$ are as in 
Proposition ~\ref{struktur}. Moreover we can find smooth cutoff
functions $\chi_j$ 
with support in $\V_j$ such that $\mu=\sum_j\chi_j\mu$. 
Then by Corollary ~\ref{nusa} there are 
\cqa s $V_j\subset \V_j\cap Z$ of codimension $\geq p+1$ such that $\chi_j\mu$ is of the form
\eqref{ballong} in $\V_j\setminus V_j$. 

Fix $j$, pick $x\in Z\setminus V_j$, let $\W\subset \V_j\setminus V_j$ be a
neighborhood of $x$, and let $\chi$ be a cutoff function with compact
support in $\W$ that is $1$ in a neighborhood of $x$. Then $\chi\chi_j\mu$
is of the form \eqref{ballong} and thus $\T(\chi\chi_j\mu)=0$ by
Lemma ~\ref{surpuppa}. Next, since $(1-\chi)\chi_j\mu=0$ in $\W$, $(i)$ implies that
$\T\big((1-\chi)\chi_j\mu\big)=0$ in $\W$. Since $\T$ is linear,  
\[
\T(\chi_j\mu)=\T(\chi\chi_j\mu)+\T\big( (1-\chi)\chi_j\mu\big ) =0
\]
in $\W$. Since $x$ was arbitrary we conclude that $\supp (\T(\chi_j\mu))\subset
V_j$. 
Now the finite union  $V=\cup_j V_j$ is a cqa of codimension $\leq d$ is a \cqa\ of dimension $\leq d$ and $\supp (\T\mu)\subset V$. 
\end{proof}

\begin{lma}\label{borg}
Given $m\in\N$, there is a constant $c_m$ such that if $\mu$ is a
pseudomeromorphic current with support on a \cqa\ of dimension
$\leq m$, then $\F^j\mu$
vanishes if $j\ge c_m$.
\end{lma}

In fact, it follows from the proof below that we can choose $c_m$ as
$2^{m+1}-1$.

\begin{proof}
First assume that $m=0$. By Example ~\ref{nolldim}, a \cqa\ of
dimension $0$ is a
variety of dimension $0$, and thus $\F\mu$ vanishes by Lemma ~\ref{formel} $(ii)$. 
It follows that the lemma holds in this case with $c_0=1$. 

Now assume that the lemma holds for $m=\ell$. Moreover, assume that
$\mu$ is a pseudomeromorphic current with support on a \cqa\ $V\subset
\U$ of
dimension $\ell+1$. Let $V'\subset V$ be a \cqa\ of
dimension $\leq \ell$ as in Lemma ~\ref{caput2}. 
We claim that $\F^{c_\ell+1}\mu$ has support on $V'$. Taking this for
granted we get that 
\[
\F^{c_\ell}(\F^{c_\ell+1}\mu)=0,
\]
by the induction hypothesis.
Thus the lemma holds for $m=\ell+1$ with
$c_{\ell+1}=2c_\ell+1$, and hence by induction for all $m$.

It remains to prove the claim. Take $x\in V\setminus V'$,  
let $\V\subset \U$ be a neighborhood of $x$ as in Lemma ~\ref{caput2}, so that
$V\cap \V\subset W=\cup W_j$, where the $W_j\subset \V$ are
submanifolds of dimension $\leq \ell+1$, and let $\chi$ be a cutoff
function with compact support in $\V$ that is $1$ in a neighborhood
$\tilde\V$ of
$x$. 
Let $\mu_j=\1_{W_j}\mu$. Then 
\[
\chi\mu = \sum_j \chi \mu_j + \nu, 
\]
where $\nu$ is a pseudomeromorphic current with
\[
\supp \nu \subset W_{\text{sing}}\cap \supp \chi=:A; 
\]
by Example ~\ref{vanlig} $A$ is
a \cqa\ of dimension $\leq \ell$.
By Lemma ~\ref{formel} ~$(i)$ $\supp (\F\nu)\subset A$, and by Lemma
~\ref{formel} ~$(iii)$ there are \cqa s $V_j\subset W$ of dimension
$\leq \ell$ such that $\supp (\F\mu_j)\subset V_j$. 
Thus, since $\F$ is linear, 
\[
\supp \big(\F(\chi \mu)\big )\subset \bigcup_j V_j \cup A=:\widetilde A.
\]
Since a finite union of \cqa s of dimension $\leq \ell$ is a \cqa\ of
dimension $\leq \ell$, $\widetilde A$ is a \cqa\ of dimension $\leq
\ell$. 
Therefore, using that the lemma holds for $m=\ell$, 
\[
\F^{c_\ell} \big (\F(\chi\mu)\big )=0. 
\]
Next, since $(1-\chi)\mu=0$ in $\tilde\V$, Lemma ~\ref{formel} ~$(i)$
gives that $\F^\kappa \big( (1-\chi)\mu\big )=0$ in $\tilde\V$ for
any $\kappa\geq 1$. 
We conclude that 
\[
\F^{c_\ell+1}\mu = \F^{c_\ell} \big (\F(\chi\mu)\big ) + \F^{c_\ell+1}
\big ((1-\chi)\mu \big ) = 0 
\]
in $\tilde\V$. Since $x$ was arbitrary
this proves the claim.  
\end{proof}

\end{document}